\documentclass[11pt,reqno]{amsart}
\usepackage{amssymb,latexsym}
\usepackage{graphicx}
\usepackage{url}		

\calclayout

\makeatletter
\newcommand{\monthyear}[1]{%
  \def\@monthyear{\uppercase{#1}}}
\newcommand{\volnumber}[1]{%
  \def\@volnumber{\uppercase{#1}}}
\AtBeginDocument{%
\def\ps@plain{\ps@empty
  \def\@oddfoot{\@monthyear \hfil \thepage}%
  \def\@evenfoot{\thepage \hfil \@volnumber}}
\def\ps@firstpage{\ps@plain}
\def\ps@headings{\ps@empty
  \def\@evenhead{%
    \setTrue{runhead}%
    \def\thanks{\protect\thanks@warning}%
    \uppercase{The Fibonacci Quarterly}\hfil}%
  \def\@oddhead{%
    \setTrue{runhead}%
    \def\thanks{\protect\thanks@warning}%
    \hfill\uppercase{Hypergeometric Template}}%
  \let\@mkboth\markboth
  \def\@evenfoot{%
    \thepage \hfil \@volnumber}%
  \def\@oddfoot{%
    \@monthyear \hfil \thepage}%
  }%
\footskip=25pt
\pagestyle{headings}%
}
\makeatother

\newcommand{\N}{{\mathbb N}}
\newcommand{\Z}{{\mathbb Z}}

\theoremstyle{plain}
\numberwithin{equation}{section}
\newtheorem{thm}{Theorem}[section]
\newtheorem{theorem}[thm]{Theorem}

\newtheorem{example}[thm]{Example}

\newtheorem{proposition}[thm]{Proposition}
\newtheorem{corollary}[thm]{Corollary}
\newtheorem{remark}[thm]{Remark}

\begin{document}
\monthyear{Month Year}
\volnumber{Volume, Number}
\setcounter{page}{1}

\title{Divisibility properties of factors of the discriminant of generalized Fibonacci numbers}
\author{Yao-Qiang Li}
\address{Institut de Math\'ematiques de Jussieu - Paris Rive Gauche\\
         Sorbonne Universit\'e - Campus Pierre et Marie Curie\\
         Paris, 75005\\
         France}
\email{yaoqiang.li@imj-prg.fr\quad yaoqiang.li@etu.upmc.fr}
\thanks{The author thanks Prof. Jean-Paul Allouche for his advices, Dr. Shuo Li for discussions, the referee for useful suggestions, and the Oversea Study Program of Guangzhou Elite Project (GEP) for financial support (JY201815).}
\address{School of Mathematics\\
         South China University of Technology\\
         Guangzhou, 510641\\
         P.R. China}
\email{scutyaoqiangli@gmail.com\quad scutyaoqiangli@qq.com}

\begin{abstract}
We study some divisibility properties related to the factors of the discriminant of the characteristic polynomial of generalized Fibonacci sequences $(G_n)_{n\ge0}$ defined by $G_0=0$, $G_1=1$ and $G_n=pG_{n-1}+qG_{n-2}$ for $n\ge2$, where $p,q$ are given integers. As corollaries, we give some divisibility properties on some well known sequences.
\end{abstract}

\maketitle

\section{Introduction}
Let $\N$ be the set of positive integers $1,2,3,\cdots$ and $\Z$ be the set of all integers. Given $p,q\in\Z$, the \textit{$<p,q>$-Fibonacci sequence} $(G_n)_{n\ge0}$ is defined by
$$G_0=0,\quad G_1=1\quad\text{and}\quad G_n=pG_{n-1}+qG_{n-2}\quad\text{for all }n\ge2.$$
For rich applications of these sequences in science and nature, see for examples \cite{K04,K19,V89}.

Let $r=r(p,q):=p^2+4q$ be the discriminant of the characteristic polynomial $x^2-px-q$ of the $<p,q>$-Fibonacci sequence $(G_n)_{n\ge0}$. For the classical Fibonacci sequence $(F_n)_{n\ge0}$ ($p=q=1$), it was shown by Kuipers and Shiue \cite{KS72} that the only moduli for which $(F_n)_{n\ge0}$ can possibly be uniformly distributed are the powers of the discriminant $r=5$. Soon Niederreiter \cite{N72} proved that $(F_n)_{n\ge0}$ is uniformly distributed mod all powers of $5$. Later the results of Niederreiter and Shiue \cite{NS77,NS80} on uniform distribution of linear recurring sequences in finite fields lead to the observation (see \cite{H04}) that over the integers, a linear recurring sequence can be uniformly distributed mod $s$ (and mod $s^k$) only if $s$ divides the discriminant of the characteristic polynomial. These results motivate the investigation of the divisors (and their powers) of the discriminant of the characteristic polynomial of generalized Fibonacci sequences in this paper. For more overviews on uniform distribution of linear recurring sequences, we refer the reader to \cite{H04}.

Throughout this paper, for $n,m\in\Z$, we use $n\mid m$, $n\nmid m$ and $(n,m)$ to denote that $n$ divides $m$ (i.e., there exists $k\in\Z$ such that $m=kn$), $n$ does not divide $m$ and the greatest common divisor of $n$ and $m$ (if $n\neq0$ or $m\neq0$) respectively. Note that $0\mid0$. An integer sequence $(a_i)_{i\ge0}$ with the property that
$$n\mid m\quad\text{implies}\quad a_n\mid a_m\quad\text{for all }n,m\ge0$$
is called a \textit{divisibility sequence}. For all $p,q\in\Z$, by induction it is not difficult to prove that the $<p,q>$-Fibonacci sequences are divisibility sequences (see for example \cite[2.2 Proposition]{HS85}). On divisibility of the terms by subscripts, we refer the reader to \cite{A91,HB74,S10,S93}. In this paper we focus on divisibility properties related to divisors of the discriminant of the characteristic polynomial of generalized Fibonacci sequences. The following theorems and corollaries are our main results.

\begin{theorem}\label{main} Let $p,q\in\Z$, $(G_n)_{n\ge0}$ be the $<p,q>$-Fibonacci sequence and $r=p^2+4q\neq0$.
\begin{itemize}
\item[(1)] For all $s\in\N$ satisfying $s\mid r$ and for all integers $k,n\ge0$, we have
$$s^kG_n\mid G_{s^kn}.$$
\item[(2)] Suppose that
$$p\text{ is odd, }(p,q)=1\text{ and }s\in\N\text{ satisfying }s\mid r$$
or
$$p\text{ is even, }(\frac{p}{2},q)=1\text{ and }s\in\N\text{ satisfying }s\mid\frac{r}{4}$$
or
$$(p,q)=1\text{ and }s\ge3\text{ is a prime satisfying }s\mid r.$$
If $3\nmid q+1$ or $3\nmid s$, then for all integers $k,n\ge0$,
$$s^k\mid n\quad\text{if and only if}\quad s^k\mid G_n.$$
\end{itemize}
\end{theorem}

\begin{theorem}\label{main2}
Let $p,q\in\Z$, $(G_n)_{n\ge0}$ be the $<p,q>$-Fibonacci sequence, $r=p^2+4q\neq0$ and $s\in\N$. Suppose that
$$p\text{ is odd, }(p,q)=1\text{ and }s\mid r$$
or
$$p\text{ is even, }(\frac{p}{2},q)=1\text{ and }s\mid \frac{r}{4}$$
or
$$(p,q)=1\text{ and }s\text{ is a prime satisfying }s\mid r.$$
\begin{itemize}
\item[(1)] For all integers $n\ge0$,
$$s\mid n\quad\text{if and only if}\quad s\mid G_n.$$
\item[(2)] If for all $t\in\N$,
$$s\nmid t\quad\text{implies}\quad s^2\nmid G_{st},$$
then for all integers $k,n\ge0$,
$$s^k\mid n\quad\text{if and only if}\quad s^k\mid G_n.$$
\end{itemize}
\end{theorem}

It is worth to note that the $<p,q>$-Fibonacci sequence $(G_n)_{n\ge0}$ studied in this paper is exactly the Lucasian sequence $(U)=(U_n)_{n\ge0}$ in \cite{W38} with the generator (characteristic polynomial) $f(x)=x^2-px-q$. If $f(x)$ is irreducible modulo a given prime, some laws of apparition of the prime in $(U_n)_{n\ge0}$ are obtained in \cite[Theorem 5.1 and 12.1]{W38}. Our results do not require $f(x)$ to be irreducible modulo a prime, but we only consider apparition of factors of the discriminant of $f(x)$. For example, let $(G_n)_{n\ge0}$ be the $<3,4>$-Fibonacci sequence. Although the generator $f(x)=x^2-3x-4$ is not irreducible modulo $5$, since $5$ is a prime factor of the discriminant of $f(x)$, by applying Theorem \ref{main2} (1), we can get the conclusion that $5$ is the unique rank of apparition (see \cite{W38} for definition) of $5$ in $(G_n)_{n\ge0}$. Besides, \cite[Theorem 9.1]{W38} shows that $s$ is a rank of apparition of any prime $s$ in $(U_n)_{n\ge0}$ which divides the discriminant of the generator $f(x)$. For the case that $(U_n)_{n\ge0}$ is the $<p,q>$-Fibonacci sequence $(G_n)_{n\ge0}$, it is straightforward to see that our Theorem \ref{main2} (1) (with the conditions $(p,q)=1$ and $s$ is a prime satisfying $s\mid r$) recover \cite[Theorem 9.1]{W38}, noting that $(p,q)\neq1$ will imply that there exists integer $m\ge2$ which divides every term of $(G_n)_{n\ge0}$ beyond a certain point (in fact, $(p,q)\mid G_n$ for all $n\ge2$), and this is an exception stated in the postil $\S$ at the bottom of the first page in \cite{W38}.

In the following, we give some corollaries according to Theorem \ref{main} and \ref{main2}.

\begin{corollary}\label{square}
Let $p,q\in\Z$, $(G_n)_{n\ge0}$ be the $<p,q>$-Fibonacci sequence, $r=p^2+4q\neq0$ and $s\in\N$. If
$$p\text{ is odd, }(p,q)=1\text{ and }s^2\mid r$$
or
$$p\text{ is even, }(\frac{p}{2},q)=1\text{ and }s^2\mid\frac{r}{4},$$
then for all integers $k,n\ge0$,
$$s^k\mid n\quad\text{if and only if}\quad s^k\mid G_n.$$
\end{corollary}

Noting that the classical Fibonacci, Pell and Jacobsthal sequences are exactly the $<1,1>$, $<2,1>$ and $<1,2>$-Fibonacci sequences respectively, Theorem \ref{main} and Corollary \ref{square} imply the following.

\begin{corollary}[Divisibility in Fibonacci, Pell and Jacobsthal sequences]\label{classical}\indent
\newline\emph{(1)} Let $(F_n)_{n\ge0}$ be the \textit{Fibonacci sequence} defined by
$$F_0=0,\quad F_1=1\quad\text{and}\quad F_n=F_{n-1}+F_{n-2}\quad\text{for all }n\ge2.$$
\begin{itemize}
\item[\emph{\textcircled{\footnotesize{1}}}] For all integers $k,n\ge0$, we have
$$5^kF_n\mid F_{5^kn}.$$
\item[\emph{\textcircled{\footnotesize{2}}}] For all integers $k,n\ge0$,
$$5^k\mid n\quad\text{if and only if}\quad 5^k\mid F_n.$$
\end{itemize}
\emph{(2)} Let $(P_n)_{n\ge0}$ be the \textit{Pell sequence} defined by
$$P_0=0,\quad P_1=1\quad\text{and}\quad P_n=2P_{n-1}+P_{n-2}\quad\text{for all }n\ge2.$$
\begin{itemize}
\item[\emph{\textcircled{\footnotesize{1}}}] For all integers $k,n\ge0$, we have
$$2^kP_n\mid P_{2^kn}.$$
\item[\emph{\textcircled{\footnotesize{2}}}] For all integers $k,n\ge0$,
$$2^k\mid n\quad\text{if and only if}\quad 2^k\mid P_n.$$
\end{itemize}
\emph{(3)} Let $(J_n)_{n\ge0}$ be the \textit{Jacobsthal sequence} defined by
$$J_0=0,\quad J_1=1\quad\text{and}\quad J_n=J_{n-1}+2J_{n-2}\quad\text{for all }n\ge2.$$
\begin{itemize}
\item[\emph{\textcircled{\footnotesize{1}}}] For all integers $k,n\ge0$, we have
$$3^kJ_n\mid J_{3^kn}.$$
\item[\emph{\textcircled{\footnotesize{2}}}] For all integers $k,n\ge0$,
$$3^k\mid n\quad\text{if and only if}\quad 3^k\mid J_n.$$
\end{itemize}
\end{corollary}

The next corollary focuses on $<p,1>$-Fibonacci sequences, which have received a lot of attention in recent years (see for examples \cite{F16,FP09,FP07,T16}).

\begin{corollary}\label{q=1} Let $p\in\Z$, $(G_n)_{n\ge0}$ be the $<p,1>$-Fibonacci sequence, $r=p^2+4$ and $s\in\N$. If
$$p\text{ is odd and }s\mid r$$
or
$$p\text{ is even and }s\mid\frac{r}{4}$$
or
$$s\ge3\text{ is a prime satisfying }s\mid r,$$
then for all integers $k,n\ge0$,
$$s^k\mid n\quad\text{if and only if}\quad s^k\mid G_n.$$
\end{corollary}

Besides, we have the following.

\begin{corollary}\label{p=1 2} Let $q\in\Z$ and $s\in\N$. Suppose that
$$(G_n)_{n\ge0}\text{ is the $<1,q>$-Fibonacci sequence and }s\mid4q+1$$
or
$$(G_n)_{n\ge0}\text{ is the $<2,q>$-Fibonacci sequence and }s\mid q+1.$$
\begin{itemize}
\item[(1)] For all integers $n\ge0$,
$$s\mid n\quad\text{if and only if}\quad s\mid G_n.$$
\item[(2)] If $3\nmid q+1$ or $3\nmid s$, then for all integers $k,n\ge0$,
$$s^k\mid n\quad\text{if and only if}\quad s^k\mid G_n.$$
\end{itemize}
\end{corollary}

We give the last corollary as follows.

\begin{corollary}\label{r} Let $p\in\Z$, $q\in\N$, $(G_n)_{n\ge0}$ be the $<p,q>$-Fibonacci sequence and $r=p^2+4q$.
\begin{itemize}
\item[(1)] If $r$ is a prime, then for all integers $k,n\ge0$,
$$r^k\mid n\quad\text{if and only if}\quad r^k\mid G_n.$$
\item[(2)] If $\frac{r}{4}$ is a prime and $p\neq0$, then for all integers $k,n\ge0$,
$$(\frac{r}{4})^k\mid n\quad\text{if and only if}\quad(\frac{r}{4})^k\mid G_n.$$
\end{itemize}
\end{corollary}

\begin{remark} More generally, for the sequence $(G^*_n)_{n\ge0}$ defined by
$$G^*_0=0,\quad G^*_1=\alpha\quad\text{and}\quad G^*_n=pG^*_{n-1}+qG^*_{n-2}\quad\text{for all }n\ge2,$$
where $\alpha,p,q\in\Z$, Theorem \ref{main} (1) still holds, since $(G^*_n)_{n\ge0}=(\alpha G_n)_{n\ge0}$.
\end{remark}

This paper is organized as follows. In Section 2, we give some examples to clarify that the detailed conditions in Theorem \ref{main}, \ref{main2} and Corollary \ref{square}, \ref{q=1}, \ref{p=1 2}, \ref{r} can not be omitted. Then we prove the main results in Section 3, and finally present further questions in Section 4.

\section{Examples}

In this section, we give some examples to clarify that the detailed conditions in Theorem \ref{main}, \ref{main2} and Corollary \ref{square}, \ref{q=1}, \ref{p=1 2}, \ref{r} can not be omitted.

\begin{example} Let $p=q=1$, $(G_n)_{n\ge0}$ be the $<1,1>$-Fibonacci sequence, $r=p^2+4q=5$ and $s=3$ ( $\nmid r$). Then $G_s=2$ and $s\nmid G_s$. It means that the condition $s\mid r$ in Theorem \ref{main}, \ref{main2}, Corollary \ref{q=1} and the condition $s\mid4q+1$ in Corollary \ref{p=1 2} can not be omitted.
\end{example}

\begin{example} Let $p=3$, $q=9$, $(G_n)_{n\ge0}$ be the $<3,9>$-Fibonacci sequence, $r=p^2+4q=45$ and $s=3$. By $G_2=3$, we get $s\mid G_2$ but $s\nmid2$. It means that:
\begin{itemize}
\item[(1)] even if $p$ is odd, $s\ge3$ is a prime $s\mid r$ and $3\nmid q+1$, the condition $(p,q)=1$ in Theorem \ref{main} (2) and Theorem \ref{main2} can not be omitted;
\item[(2)] even if $p$ is odd and $s^2\mid r$, the condition $(p,q)=1$ in Corollary \ref{square} can not be omitted.
\end{itemize}
\end{example}

\begin{example} Let $p=4$, $q=1$, $(G_n)_{n\ge0}$ be the $<4,1>$-Fibonacci sequence and $r=p^2+4q=20$.
\newline(1) Let $s=20$. By simple calculation we get $G_{10}=416020$ and $s\mid G_{10}$ but $s\nmid10$. It means that:
\begin{itemize}
\item[\textcircled{\footnotesize{1}}] even if $(p,q)=1$, $s\mid r$ and $3\nmid q+1$, the condition that $p$ is odd in Theorem \ref{main} (2), Theorem \ref{main2} and Corollary \ref{q=1} can not be omitted;
\item[\textcircled{\footnotesize{2}}] even if $p$ is even, $(\frac{p}{2},q)=1$, $s\mid r$ and $3\nmid q+1$, the condition $s\mid\frac{r}{4}$ in Theorem \ref{main} (2), Theorem \ref{main2} and Corollary \ref{q=1} can not be omitted;
\item[\textcircled{\footnotesize{3}}] even if $(p,q)=1$, $s\ge3$ satisfies $s\mid r$ and $3\nmid q+1$, the condition that $s$ is a prime in Theorem \ref{main} (2), Theorem \ref{main2} and Corollary \ref{q=1} can not be omitted.
\end{itemize}
(2) Let $s=2$. By $G_2=4$ we get $s^2\mid G_2$ but $s^2\nmid2$. It means that:
\begin{itemize}
\item[\textcircled{\footnotesize{1}}] even if $(p,q)=1$, $s$ is a prime satisfying $s\mid r$ and $3\nmid q+1$, the condition $s\ge3$ in Theorem \ref{main} (2) and Corollary \ref{q=1} can not be omitted;
\item[\textcircled{\footnotesize{2}}] even if $(p,q)=1$ and $s^2\mid r$, the condition that $p$ is odd in Corollary \ref{square} can not be omitted.
\end{itemize}
\end{example}

\begin{example} Let $p=q=4$, $(G_n)_{n\ge0}$ be the $<4,4>$-Fibonacci sequence and $r=p^2+4q=32$.
\begin{itemize}
\item[(1)] Let $s=4$. By $G_2=4$ we get $s\mid G_2$ but $s\nmid2$. It means that even if $p$ is even, $s\mid\frac{r}{4}$ and $3\nmid q+1$, the condition $(\frac{p}{2},q)=1$ in Theorem \ref{main} (2) and Theorem \ref{main2} can not be omitted.
\item[(2)] Let $s=2$. By $G_3=20$ we get $s\mid G_3$ but $s\nmid3$. It means that even if $p$ is even and $s^2\mid\frac{r}{4}$, the condition $(\frac{p}{2},q)=1$ in Corollary \ref{square} can not be omitted.
\end{itemize}
\end{example}

\begin{example} Let $p=5$, $q=2$, $(G_n)_{n\ge0}$ be the $<5,2>$-Fibonacci sequence, $r=p^2+4q=33$ and $s=3$. By $G_3=27$ we get $s^2\mid G_3$ but $s^2\nmid3$. It means that even if $p$ is odd, $(p,q)=1$, $s\ge3$ is a prime and $s\mid r$, the condition  $3\nmid q+1$ or $3\nmid s$ in Theorem \ref{main} (2) and the condition $s^2\mid r$ in Corollary \ref{square} can not be omitted.
\end{example}

\begin{example} Let $p=2$, $q=5$, $(G_n)_{n\ge0}$ be the $<2,5>$-Fibonacci sequence, $r=p^2+4q=24$ and $s=3$. By $G_3=9$ we get $s^2\mid G_3$ but $s^2\nmid3$. It means that:
\begin{itemize}
\item[(1)] even if $p$ is even, $(\frac{p}{2},q)=1$ and $s\mid\frac{r}{4}$, the condition $3\nmid q+1$ or $3\nmid s$ in Theorem \ref{main} (2) and the condition $s^2\mid\frac{r}{4}$ in Corollary \ref{square} can not be omitted;
\item[(2)] even if $s\mid q+1$, the condition $3\nmid q+1$ or $3\nmid s$ in Corollary \ref{p=1 2} (2) can not be omitted.
\end{itemize}
\end{example}

\begin{example} Let $p=q=2$, $(G_n)_{n\ge0}$ be the $<2,2>$-Fibonacci sequence and $r=p^2+4q=12$.
\begin{itemize}
\item[(1)] By simple calculation we get $G_6=120$ and $r\mid G_6$ but $r\nmid6$. It means that the condition that $r$ is a prime in Corollary \ref{r} (1) can not be omitted.
\item[(2)] Let $s=2$. By $G_3=6$ we get $s\mid G_3$ but $s\nmid3$. It means that the condition $s\mid q+1$ in Corollary \ref{p=1 2} can not be omitted.
\end{itemize}
\end{example}

\begin{example} Let $p=4$, $q=2$, $(G_n)_{n\ge0}$ be the $<4,2>$-Fibonacci sequence and $r=p^2+4q=24$. Then $\frac{r}{4}=6$. By $G_3=18$ we get $\frac{r}{4}\mid G_3$ but $\frac{r}{4}\nmid3$. It means that even if $4\mid r$ and $p\neq0$, the condition that $\frac{r}{4}$ is a prime in Corollary \ref{r} (2) can not be omitted.
\end{example}

\begin{example} Let $p=1$, $q=8$, $(G_n)_{n\ge0}$ be the $<1,8>$-Fibonacci sequence and $s=3$. By $G_3=9$ we get $s^2\mid G_3$ but $s^2\nmid3$. It means that even if $s\mid4q+1$, the condition $3\nmid q+1$ or $3\nmid s$ in Corollary \ref{p=1 2} (2) can not be omitted.
\end{example}

\begin{example} Let $p=0$, $q=2$, $(G_n)_{n\ge0}$ be the $<0,2>$-Fibonacci sequence and $r=p^2+4q=8$. Then $\frac{r}{4}=2$. By $G_3=2$ we get $\frac{r}{4}\mid G_3$ but $\frac{r}{4}\nmid3$. It means that even if $\frac{r}{4}$ is a prime, the condition $p\neq0$ in Corollary \ref{r} (2) can not be omitted.
\end{example}

\begin{example} Let $p=5$, $q=-5$, $(G_n)_{n\ge0}$ be the $<5,-5>$-Fibonacci sequence and $r=p^2+4q=5$. By $G_2=5$ we get $r\mid G_2$ but $r\nmid2$. It means that even if $q\in\Z$ and $r$ is a prime, the condition $q\in\N$ in Corollary \ref{r} can not be omitted for the statement (1).
\end{example}

\begin{example} Let $p=4$, $q=-2$, $(G_n)_{n\ge0}$ be the $<4,-2>$-Fibonacci sequence and $r=p^2+4q=8$. Then $\frac{r}{4}=2$. By $G_3=14$ we get $\frac{r}{4}\mid G_3$ but $\frac{r}{4}\nmid3$. It means that even if $q\in\Z$, $\frac{r}{4}$ is a prime and $p\neq0$, the condition $q\in\N$ in Corollary \ref{r} can not be omitted for the statement (2).
\end{example}

\section{Proof of the main results}

The following proposition, which says that generalized Fibonacci sequences are all divisibility sequences, follows from \cite[2.2 Proposition]{HS85} (see also \cite[Theorem IV]{C13}).

\begin{proposition}\label{1} Let $p,q\in\Z$ and $(G_n)_{n\ge0}$ be the $<p,q>$-Fibonacci sequence. Then for all integers $k,n\ge0$, we have $G_n\mid G_{kn}$.
\end{proposition}

First we prove Theorem \ref{main} (1), then Theorem \ref{main2}, then Theorem \ref{main} (2), and finally the corollaries.

\begin{proof}[Proof of Theorem \ref{main} (1)]
By the Binet formula (see for examples \cite[Theorem 2]{YT12} and \cite[2.5 Corollary]{SK13}), for all integers $n\ge0$, we have
\begin{eqnarray}\label{Gn}
G_n=\frac{(p+\sqrt{p^2+4q})^n-(p-\sqrt{p^2+4q})^n}{2^n\sqrt{p^2+4q}}=\frac{(p+\sqrt{r})^n-(p-\sqrt{r})^n}{2^n\sqrt{r}},
\end{eqnarray}
where $r$ can be negative and $\sqrt{r}$ is a complex number. For all integers $n\ge0$, let
\begin{eqnarray}\label{AnBn def}
A_n:=\frac{(p+\sqrt{r})^n+(p-\sqrt{r})^n}{2}\quad\text{and}\quad B_n:=\frac{(p+\sqrt{r})^n-(p-\sqrt{r})^n}{2\sqrt{r}}.
\end{eqnarray}
Then $A_n$ and $B_n$ are both integers,
\begin{eqnarray}\label{AnBn}
\left\{\begin{array}{ll}
A_n+B_n\sqrt{r}=(p+\sqrt{r})^n\\
A_n-B_n\sqrt{r}=(p-\sqrt{r})^n
\end{array}\right.
\end{eqnarray}
and
$$G_n=\frac{B_n}{2^{n-1}}.$$
For all integers $n\ge0$ and $s\ge1$, by
$$A_{sn}+B_{sn}\sqrt{r}=(p+\sqrt{r})^{sn}=(A_n+B_n\sqrt{r})^s$$
we get
\begin{eqnarray*}
B_{sn}\sqrt{r}=\left\{\begin{array}{ll}
\binom{s}{1}A_n^{s-1}B_n\sqrt{r}+\binom{s}{3}A_n^{s-3}(B_n\sqrt{r})^3+\cdots+\binom{s}{s}(B_n\sqrt{r})^s&\text{if $s$ is odd,}\\
\binom{s}{1}A_n^{s-1}B_n\sqrt{r}+\binom{s}{3}A_n^{s-3}(B_n\sqrt{r})^3+\cdots+\binom{s}{s-1}A_n(B_n\sqrt{r})^{s-1}&\text{if $s$ is even,}
\end{array}\right.
\end{eqnarray*}
where $\binom{s}{t}:=\frac{s!}{(s-t)!\cdot t!}$ for all $t\in\{0,1,\cdots,s\}$, and then
\begin{eqnarray}\label{Bsn}
B_{sn}=\left\{\begin{array}{ll}
sA_n^{s-1}B_n+\binom{s}{3}A_n^{s-3}B_n^3r+\cdots+\binom{s}{s}B_n^sr^{\frac{s-1}{2}}&\text{if $s$ is odd,}\\
sA_n^{s-1}B_n+\binom{s}{3}A_n^{s-3}B_n^3r+\cdots+\binom{s}{s-1}A_nB_n^{s-1}r^\frac{s-2}{2}&\text{if $s$ is even.}
\end{array}\right.
\end{eqnarray}
In the following we prove that for all $s\in\N$ satisfying $s\mid r$, we have $s^kG_n\mid G_{s^kn}$ for all integers $k,n\ge0$. Obviously we only need to consider $k,n\ge1$ and $s\ge2$. For $G_n=0$, by Proposition \ref{1} we get $G_{s^kn}=0$ and then $s^kG_n\mid G_{s^kn}$ follows immediately. In the following it suffices to consider $G_n\neq0$, which implies $B_n\neq0$.
\newline\textcircled{\footnotesize{1}} Prove $sG_n\mid G_{sn}$ for all $n\in\N$.
\begin{itemize}
\item[i)] Suppose that $s$ is odd.
\newline On the one hand, (\ref{Bsn}) implies that $\frac{B_{sn}}{B_n}$ is an integer and $s\mid\frac{B_{sn}}{B_n}$ (applying $s\mid r$). On the other hand, by Proposition \ref{1} we get $G_n\mid G_{sn}$, which implies $2^{(s-1)n}\mid\frac{B_{sn}}{B_n}$. It follows from $(s,2)=1$ that $s2^{(s-1)n}\mid\frac{B_{sn}}{B_n}$. Thus $sG_n\mid G_{sn}$.
\item[ii)] Suppose that $s$ is even.
\newline\textcircled{\footnotesize{a}} Prove $2G_n\mid G_{2n}$, i.e., $2^{n+1}B_n\mid B_{2n}$ for all $n\in\N$.
\newline Since (\ref{Bsn}) implies $B_{2n}=2A_nB_n$, it suffices to prove $2^n\mid A_n$. In fact, by (\ref{AnBn}) we get $(A_n+B_n\sqrt{r})(A_n-B_n\sqrt{r})=(p+\sqrt{r})^n(p-\sqrt{r})^n$ and then $A_n^2-B_n^2r=(p^2-r)^n$. It follows from $r=p^2+4q$ and $B_n=2^{n-1}G_n$ that
\begin{eqnarray}\label{An}
A_n^2=4^{n-1}rG_n^2+(-4q)^n.
\end{eqnarray}
Since $2\mid s$ and $s\mid r$ implies $2\mid r$, by $r=p^2+4q$, we get $2\mid p$ and then $4\mid r$. It follows from (\ref{An}) that $4^n\mid A_n^2$, which is equivalent to $2^n\mid A_n$.
\newline\textcircled{\footnotesize{b}} Prove $sG_n\mid G_{sn}$ for all $n\in\N$.
\newline Since $s$ is even, there exist $a,t\in\N$ such that $s=2^at$ where $t$ is odd. By \textcircled{\footnotesize{a}} we get
$$2G_{tn}\mid G_{2tn},\quad 2G_{2tn}\mid G_{2^2tn},\quad 2G_{2^2tn}\mid G_{2^3tn},\quad\cdots,\quad 2G_{2^{a-1}tn}\mid G_{2^atn},$$
which imply
$$2^aG_{tn}\mid 2^{a-1}G_{2tn},\quad 2^{a-1}G_{2tn}\mid 2^{a-2}G_{2^2tn},\quad 2^{a-2}G_{2^2tn}\mid 2^{a-3}G_{2^3tn},\quad\cdots,\quad 2G_{2^{a-1}tn}\mid G_{2^atn}.$$
Thus $2^aG_{tn}\mid G_{2^atn}$. Since $t\mid r$ and $t$ is odd, by i) we get $tG_n\mid G_{tn}$ and then $2^atG_n\mid2^aG_{tn}$. Therefore $2^atG_n\mid G_{2^atn}$, i.e., $sG_n\mid G_{sn}$.
\end{itemize}
\textcircled{\footnotesize{2}} Prove $s^kG_n\mid G_{s^kn}$ for all $n,k\in\N$. In fact, by \textcircled{\footnotesize{1}} we get
$$sG_n\mid G_{sn},\quad sG_{sn}\mid G_{s^2n},\quad sG_{s^2n}\mid G_{s^3n},\quad\cdots,\quad sG_{s^{k-1}n}\mid G_{s^kn},$$
which imply
$$s^kG_n\mid s^{k-1}G_{sn},\quad s^{k-1}G_{sn}\mid s^{k-2}G_{s^2n},\quad s^{k-2}G_{s^2n}\mid s^{k-3}G_{s^3n},\quad\cdots,\quad sG_{s^{k-1}n}\mid G_{s^kn}.$$
Therefore $s^kG_n\mid G_{s^kn}$.
\end{proof}

\begin{proof}[Proof of Theorem \ref{main2}]\indent
\newline(\underline{Case 1}) Suppose that $p$ is odd, $(p,q)=1$ and $s\mid r$.
\newline First we prove $(s,p)=1$. It suffices to prove $(r,p)=1$. Let $k=(r,p)$. Then there exist $a,b\in\Z$ such that $r=ak$ and $p=bk$. It follows from $r=p^2+4q$ that $q=\frac{(a-b^2k)k}{4}$. Since $p$ is odd, $k$ must be odd. By $q\in\N$ we get $\frac{a-b^2k}{4}\in\N$. It follows from $(p,q)=1$ that $k=1$.

Besides, since $r=p^2+4q$ is odd, we know that $s$ is also odd. For $s=1$, the conclusions are obviously true. We only need to consider $s\ge3$ in the following.
\begin{itemize}
\item[(1)] Prove that for all integers $n\ge0$, $s\mid n$ if and only if $s\mid G_n$.
\newline$\boxed{\Rightarrow}$ follows directly from Theorem \ref{main} (1).
\newline$\boxed{\Leftarrow}$ It suffices to consider $n\ge1$. Suppose $s\mid G_n$. Then $s\mid B_n$. Since (\ref{AnBn def}) implies
\begin{eqnarray}\label{Bn}
B_n=\left\{\begin{array}{ll}
np^{n-1}+\binom{n}{3}p^{n-3}r+\binom{n}{5}p^{n-5}r^2+\cdots+\binom{n}{n}r^{\frac{n-1}{2}}&\text{if $n$ is odd,}\\
np^{n-1}+\binom{n}{3}p^{n-3}r+\binom{n}{5}p^{n-5}r^2+\cdots+\binom{n}{n-1}pr^\frac{n-2}{2}&\text{if $n$ is even,}
\end{array}\right.
\end{eqnarray}
it follows from $s\mid B_n$ and $s\mid r$ that $s\mid np^{n-1}$. By $(s,p)=1$ we get $s\mid n$.
\item[(2)] Suppose that for all $t\in\N$, $s\nmid t$ implies $s^2\nmid G_{st}$. We prove that for all integers $k,n\ge0$, $s^k\mid n$ if and only if $s^k\mid G_n$.
\newline$\boxed{\Rightarrow}$ follows directly from Theorem \ref{main} (1).
\newline$\boxed{\Leftarrow}$ \textcircled{\footnotesize{1}} First we prove that for all $t\in\N$ such that $s\nmid t$, we have $s^{k+1}\nmid G_{s^kt}$ for all $k\ge0$ by induction.
\begin{itemize}
\item[i)] For $k=0$, $s\nmid G_t$ follows from (1) $\boxed{\Leftarrow}$.
\item[ii)] For $k=1$, $s^2\nmid G_{st}$ follows from the condition that $s\nmid t$ implies $s^2\nmid G_{st}$.
\item[iii)] Assume that for some $k\in\N$, we have already had $s^{k+1}\nmid G_{s^kt}$ for all $t\in\N$ satisfying $s\nmid t$. It suffices to prove $s^{k+2}\nmid G_{s^{k+1}t}$ in the following by contradiction. Assume $s^{k+2}\mid G_{s^{k+1}t}$ for some $t\in\N$ satisfying $s\nmid t$. Then $s^{k+2}\mid B_{s^{k+1}t}$. Since $s$ is odd, by (\ref{Bsn}) we get
\begin{eqnarray}\label{Bsskt}
B_{s^{k+1}t}=B_{s(s^kt)}=sA_{s^kt}^{s-1}B_{s^kt}+\binom{s}{3}A_{s^kt}^{s-3}B_{s^kt}^3r+\cdots+\binom{s}{s}B_{s^kt}^sr^{\frac{s-1}{2}}.
\end{eqnarray}
Noting that Theorem \ref{main} (1) implies $s^k\mid G_{s^kt}$, we get $s^k\mid B_{s^kt}$ and then $s^{k+2}\mid B_{s^kt}^3$. By $s^{k+2}\mid B_{s^{k+1}t}$ and (\ref{Bsskt}) we get $s^{k+2}\mid sA_{s^kt}^{s-1}B_{s^kt}$ and then
\begin{eqnarray}\label{s k+1}
s^{k+1}\mid A_{s^kt}^{s-1}B_{s^kt}.
\end{eqnarray}
Since (\ref{AnBn def}) implies
$$A_{s^kt}=p^{s^kt}+c(p,r,s,k,t)\quad\text{where }r\mid c(p,r,s,k,t),$$
by $s\mid r$ and $(s,p)=1$, we get $(s,A_{s^kt})=1$. It follows from (\ref{s k+1}) that $s^{k+1}\mid B_{s^kt}$. By $(s,2)=1$ and $B_{s^kt}=2^{s^kt-1}G_{s^kt}$, we get $s^{k+1}\mid G_{s^kt}$, which contradicts the inductive hypothesis.
\end{itemize}
\textcircled{\footnotesize{2}} Let $k,n\ge0$ be integers and suppose $s^k\mid G_n$. We need to prove $s^k\mid n$. It suffices to consider $k,n\ge1$. Let $l\ge0$ and $t\ge1$ be integers such that $n=s^lt$ with $s\nmid t$. By \textcircled{\footnotesize{1}} we get $s^{l+1}\nmid G_n$. It follows from $s^k\mid G_n$ that $k\le l$, which implies $s^k\mid s^lt$ ($=n$).
\end{itemize}
(\underline{Case 2}) Suppose that $p$ is even, $(\frac{p}{2},q)=1$ and $s\mid\frac{r}{4}$.
\newline First we prove $(s,\frac{p}{2})=1$. It suffices to prove $(\frac{r}{4},\frac{p}{2})=1$. Let $k=(\frac{r}{4},\frac{p}{2})$. Then there exists $a,b\in\Z$ such that $\frac{r}{4}=ak$ and $\frac{p}{2}=bk$. It follows from $r=p^2+4q$ that $q=(a-b^2k)k$. By $(\frac{p}{2},q)=1$ we get $k=1$.
\begin{itemize}
\item[(1)] Prove that for all integers $n\ge0$, $s\mid n$ if and only if $s\mid G_n$.
\newline$\boxed{\Rightarrow}$ follows directly from Theorem \ref{main} (1).
\newline$\boxed{\Leftarrow}$ It suffices to consider $n\ge1$. Suppose $s\mid G_n$. By (\ref{Gn}) we get
\begin{eqnarray}\label{Gn expan}
G_n=\left\{\begin{array}{ll}
n(\frac{p}{2})^{n-1}+\binom{n}{3}(\frac{p}{2})^{n-3}\cdot\frac{r}{4}+\binom{n}{5}(\frac{p}{2})^{n-5}(\frac{r}{4})^2+\cdots+\binom{n}{n}(\frac{r}{4})^{\frac{n-1}{2}}&\text{if $n$ is odd,}\\
n(\frac{p}{2})^{n-1}+\binom{n}{3}(\frac{p}{2})^{n-3}\cdot\frac{r}{4}+\binom{n}{5}(\frac{p}{2})^{n-5}(\frac{r}{4})^2+\cdots+\binom{n}{n-1}\frac{p}{2}(\frac{r}{4})^\frac{n-2}{2}&\text{if $n$ is even.}
\end{array}\right.
\end{eqnarray}
It follows from $s\mid G_n$ and $s\mid\frac{r}{4}$ that $s\mid n(\frac{p}{2})^{n-1}$. By $(s,\frac{p}{2})=1$ we get $s\mid n$.
\item[(2)] Suppose that for all $t\in\N$, $s\nmid t$ implies $s^2\nmid G_{st}$. We prove that for all $k,n\ge0$, $s^k\mid n$ if and only if $s^k\mid G_n$.
\newline$\boxed{\Rightarrow}$ follows directly from Theorem \ref{main} (1).
\newline$\boxed{\Leftarrow}$ \textcircled{\footnotesize{1}} First we prove that for all $t\in\N$ such that $s\nmid t$, we have $s^{k+1}\nmid G_{s^kt}$ for all $k\ge0$ by induction.
\begin{itemize}
\item[i)] For $k=0$, $s\nmid G_t$ follows from (1) $\boxed{\Leftarrow}$.
\item[ii)] For $k=1$, $s^2\nmid G_{st}$ follows from the condition that $s\nmid t$ implies $s^2\nmid G_{st}$.
\item[iii)] Assume that for some $k\in\N$, we have already had $s^{k+1}\nmid G_{s^kt}$ for all $t\in\N$ satisfying $s\nmid t$. It suffices to prove $s^{k+2}\nmid G_{s^{k+1}t}$ in the following by contradiction. Assume $s^{k+2}\mid G_{s^{k+1}t}$ for some $t\in\N$ satisfying $s\nmid t$. By (\ref{Bsn}) we get
\begin{eqnarray}\label{Bsskt2}
G_{s^{k+1}t}=\frac{B_{s(s^kt)}}{2^{s^{k+1}t-1}}=\left\{\begin{array}{rl}
s(\frac{A_{s^kt}}{2^{s^kt}})^{s-1}\cdot\frac{B_{s^kt}}{2^{s^kt-1}}+\binom{s}{3}(\frac{A_{s^kt}}{2^{s^kt}})^{s-3}(\frac{B_{s^kt}}{2^{s^kt-1}})^3\cdot\frac{r}{4}&\\
+\cdots+\binom{s}{s}(\frac{B_{s^kt}}{2^{s^kt-1}})^s(\frac{r}{4})^{\frac{s-1}{2}}&\text{if $s$ is odd,}\\
s(\frac{A_{s^kt}}{2^{s^kt}})^{s-1}\cdot\frac{B_{s^kt}}{2^{s^kt-1}}+\binom{s}{3}(\frac{A_{s^kt}}{2^{s^kt}})^{s-3}(\frac{B_{s^kt}}{2^{s^kt-1}})^3\cdot\frac{r}{4}&\\
+\cdots+\binom{s}{s-1}\frac{A_{s^kt}}{2^{s^kt}}(\frac{B_{s^kt}}{2^{s^kt-1}})^{s-1}(\frac{r}{4})^{\frac{s-2}{2}}&\text{if $s$ is even.}\\
\end{array}\right.
\end{eqnarray}
It follows from (\ref{AnBn def}) that
\begin{eqnarray}\label{integer}
\frac{A_n}{2^n}=\left\{\begin{array}{ll}
\binom{n}{0}(\frac{p}{2})^n+\binom{n}{2}(\frac{p}{2})^{n-2}\cdot\frac{r}{4}+\cdots+\binom{n}{n}(\frac{r}{4})^{\frac{n}{2}}&\text{if $n$ is even,}\\
\binom{n}{0}(\frac{p}{2})^n+\binom{n}{2}(\frac{p}{2})^{n-2}\cdot\frac{r}{4}+\cdots+\binom{n}{n-1}\frac{p}{2}(\frac{r}{4})^{\frac{n-1}{2}}&\text{if $n$ is odd,}
\end{array}\right.
\end{eqnarray}
for all $n\ge0$, which implies that $\frac{A_{s^kt}}{2^{s^kt}}$ is an integer. Since Theorem \ref{main} (1) implies $s^k\mid G_{s^kt}$, we get $s^{k+2}\mid G_{s^kt}^3$ ($=(\frac{B_{s^kt}}{2^{s^kt-1}})^3$). It follows from $s^{k+2}\mid G_{s^{k+1}t}$ and (\ref{Bsskt2}) that $s^{k+2}\mid s(\frac{A_{s^kt}}{2^{s^kt}})^{s-1}\cdot\frac{B_{s^kt}}{2^{s^kt-1}}$ and then
\begin{eqnarray}\label{s k+1 div}
s^{k+1}\mid(\frac{A_{s^kt}}{2^{s^kt}})^{s-1}\cdot\frac{B_{s^kt}}{2^{s^kt-1}}.
\end{eqnarray}
Since (\ref{integer}) implies
$$\frac{A_{s^kt}}{2^{s^kt}}=(\frac{p}{2})^{s^kt}+c(p,r,s,k,t)\quad\text{where }\frac{r}{4}\mid c(p,r,s,k,t),$$
by $s\mid\frac{r}{4}$ and $(s,\frac{p}{2})=1$, we get $(s,\frac{A_{s^kt}}{2^{s^kt}})=1$. It follows from (\ref{s k+1 div}) that $s^{k+1}\mid \frac{B_{s^kt}}{2^{s^kt-1}}$ ($=G_{s^kt}$), which contradicts the inductive hypothesis.
\end{itemize}
\textcircled{\footnotesize{2}} In the same way as the proof of (\underline{Case 1}) (2) $\boxed{\Leftarrow}$ \textcircled{\footnotesize{2}}, we know that for all integers $k,n\ge0$, $s^k\mid G_n$ implies $s^k\mid n$.
\end{itemize}
(\underline{Case 3}) Suppose that $(p,q)=1$ and $s$ is a prime satisfying $s\mid r$.
\newline If $p$ is odd, the conclusions follow immediately from (\underline{Case 1}). We only need to consider that $p$ is even in the following. By $(p,q)=1$ we get $(\frac{p}{2},q)=1$. Since $2\mid p$ implies $4\mid r$, by $s\mid r$ we get $s\mid2^2\cdot\frac{r}{4}$. Noting that $s$ is a prime, it follows that $s\mid 2$ or $s\mid\frac{r}{4}$. If $s\mid\frac{r}{4}$, the conclusions follow immediately from (\underline{Case 2}). In the following, we only need to consider $s\nmid\frac{r}{4}$ and $s\mid2$, which imply $s=2$ and $2\nmid\frac{r}{4}$, i.e., $2\nmid(\frac{p}{2})^2+q$. Since $(p,q)=1$ and $2\mid p$ imply $2\nmid q$, we get $2\mid(\frac{p}{2})^2$ and then $2\mid\frac{p}{2}$.
\begin{itemize}
\item[(1)] Prove that for all $n\ge0$, $2\mid n$ if and only if $2\mid G_n$.
\newline$\boxed{\Rightarrow}$ follows directly from Theorem \ref{main} (1).
\newline$\boxed{\Leftarrow}$ Suppose $2\mid G_n$. By (\ref{Gn expan}), $2\mid\frac{p}{2}$ and $2\nmid\frac{r}{4}$, we know that $n$ must be even.
\item[(2)] Suppose $2^2\nmid G_{2t}$ for all odd $t\in\N$. We prove that for all integers $k,n\ge0$, $2^k\mid n$ if and only if $2^k\mid G_n$.
\newline$\boxed{\Rightarrow}$ follows directly from Theorem \ref{main} (1).
\newline$\boxed{\Leftarrow}$ \textcircled{\footnotesize{1}} First we prove that for all odd $t\in\N$, we have $2^{k+1}\nmid G_{2^kt}$ for all $k\ge0$ by induction.
\begin{itemize}
\item[i)] For $k=0$, $2\nmid G_t$ follows from (1) $\boxed{\Leftarrow}$.
\item[ii)] For $k=1$, $2^2\nmid G_{2t}$ follows from the assumption for odd $t\in\N$.
\item[iii)] Assume that for some $k\in\N$, we have already had $2^{k+1}\nmid G_{2^kt}$ for all odd $t\in\N$. It suffices to prove $2^{k+2}\nmid G_{2^{k+1}t}$ in the following by contradiction. Assume $2^{k+2}\mid G_{2^{k+1}t}$ for some odd $t\in\N$. Since (\ref{Bsskt2}) implies $G_{2^{k+1}t}=2\cdot\frac{A_{2^kt}}{2^{2^kt}}\cdot\frac{B_{2^kt}}{2^{2^kt-1}}=2\cdot\frac{A_{2^kt}}{2^{2^kt}}\cdot G_{2^kt}$, we get
$$2^{k+1}\mid\frac{A_{2^kt}}{2^{2^kt}}\cdot G_{2^kt}$$
where $\frac{A_{2^kt}}{2^{2^kt}}$ is an integer. Since (\ref{integer}), $2\mid\frac{p}{2}$ and $2\nmid\frac{r}{4}$ imply $2\nmid\frac{A_{2^kt}}{2^{2^kt}}$, we get $2^{k+1}\mid G_{2^kt}$, which contradicts the inductive hypothesis.
\end{itemize}
\textcircled{\footnotesize{2}} In the same way as the proof of (\underline{Case 1}) (2) $\boxed{\Leftarrow}$ \textcircled{\footnotesize{2}}, we know that for all integers $k,n\ge0$, $2^k\mid G_n$ implies $2^k\mid n$.
\end{itemize}
\end{proof}

\begin{proof}[Proof of Theorem \ref{main} (2)]\indent
\newline\textcircled{\footnotesize{1}} Suppose that $p$ is odd, $(p,q)=1$, $s\in\N$ satisfying $s\mid r$, and $3\nmid q+1$ or $3\nmid s$. By Theorem \ref{main2} (2), it suffices to prove that for all $t\in\N$ satisfying $s\nmid t$, we have $s^2\nmid G_{st}$.
\newline(By contradiction) Assume $s^2\mid G_{st}$. Then $s^2\mid B_{st}$. Recall from the proof of Theorem \ref{main2} (\underline{Case 1}) that $(s,p)=1$ and $s$ is odd. Since $s\nmid t$ implies $s\neq1$, it follows that $s\ge3$. By (\ref{Bsn}) we get
$$B_{st}=sA_t^{s-1}B_t+\binom{s}{3}A_t^{s-3}B_t^3r+\binom{s}{5}A_t^{s-5}B_t^5r^2+\cdots+\binom{s}{s}B_t^sr^{\frac{s-1}{2}}.$$
It follows from $s^2\mid B_{st}$ and $s\mid r$ that
\begin{eqnarray}\label{s divides}
s\mid A_t^{s-1}B_t+\binom{s}{3}A_t^{s-3}B_t^3\cdot\frac{r}{s}.
\end{eqnarray}
Since (\ref{AnBn def}) implies
$$A_t=p^t+c_1(p,r,t)\quad\text{where }r\mid c_1(p,r,t)$$
and
$$B_t=tp^{t-1}+c_2(p,r,t)\quad\text{where }r\mid c_2(p,r,t),$$
by $s\mid r$ and (\ref{s divides}) we get
$$s\mid(p^t)^{s-1}(tp^{t-1})+\binom{s}{3}(p^t)^{s-3}(tp^{t-1})^3\cdot\frac{r}{s},$$
$$\text{i.e.,}\quad s\mid p^{st-3}\Big(tp^2+t^3\cdot\frac{r(s-1)(s-2)}{6}\Big).$$
It follows from $(s,p)=1$ that
\begin{eqnarray}\label{s div}
s\mid tp^2+t^3\cdot\frac{r(s-1)(s-2)}{6}.
\end{eqnarray}
Recall the condition $3\nmid q+1$ or $3\nmid s$. If $3\nmid s$, then $3\mid s-1$ or $3\mid s-2$. It follows from $3\mid(s-1)(s-2)$ and $2\mid(s-1)(s-2)$ that $\frac{(s-1)(s-2)}{6}\in\N$. By (\ref{s div}) and $s\mid r$, we get $s\mid tp^2$. It follows from $(s,p)=1$ that $s\mid t$, which contradicts $s\nmid t$. Thus we only consider $3\mid s$ in the following. By $(s,p)=1$ we get $3\nmid p$, and then $p\equiv\pm1$ mod $3$. It follows that $r=p^2+4q\equiv1+q$ mod $3$. Since $3\mid s$ implies $3\mid r$, we get $3\mid q+1$. This contradicts the condition $3\nmid q+1$ or $3\nmid s$.
\newline\textcircled{\footnotesize{2}} Suppose that $p$ is even, $(\frac{p}{2},q)=1$, $s\in\N$ satisfying $s\mid\frac{r}{4}$, and $3\nmid q+1$ or $3\nmid s$. By Theorem \ref{main2} (2), it suffices to prove that for all $t\in\N$ satisfying $s\nmid t$, we have $s^2\nmid G_{st}$.
\newline(By contradiction) Assume $s^2\mid G_{st}$ ($=\frac{B_{st}}{2^{st-1}}$). Recall from the proof of Theorem \ref{main2} (\underline{Case 2}) that $(s,\frac{p}{2})=1$. By (\ref{Bsn}) we get
\begin{eqnarray}\label{Bst}
G_{st}=\left\{\begin{array}{rl}
s(\frac{A_t}{2^t})^{s-1}\cdot\frac{B_t}{2^{t-1}}+\binom{s}{3}(\frac{A_t}{2^t})^{s-3}(\frac{B_t}{2^{t-1}})^3\cdot\frac{r}{4}+\binom{s}{5}(\frac{A_t}{2^t})^{s-5}(\frac{B_t}{2^{t-1}})^5(\frac{r}{4})^2&\\
+\cdots+\binom{s}{s}(\frac{B_t}{2^{t-1}})^s(\frac{r}{4})^{\frac{s-1}{2}}&\text{if $s$ is odd,}\\
s(\frac{A_t}{2^t})^{s-1}\cdot\frac{B_t}{2^{t-1}}+\binom{s}{3}(\frac{A_t}{2^t})^{s-3}(\frac{B_t}{2^{t-1}})^3\cdot\frac{r}{4}+\binom{s}{5}(\frac{A_t}{2^t})^{s-5}(\frac{B_t}{2^{t-1}})^5(\frac{r}{4})^2&\\
+\cdots+\binom{s}{s-1}\frac{A_t}{2^t}(\frac{B_t}{2^{t-1}})^{s-1}(\frac{r}{4})^{\frac{s-2}{2}}&\text{if $s$ is even.}\\
\end{array}\right.
\end{eqnarray}
Since (\ref{integer}) implies that $\frac{A_t}{2^t}$ is an integer, it follows from (\ref{Bst}), $s^2\mid G_{st}$ and $s\mid\frac{r}{4}$ that
\begin{eqnarray}\label{s div r}
s\mid(\frac{A_t}{2^t})^{s-1}\cdot\frac{B_t}{2^{t-1}}+\binom{s}{3}(\frac{A_t}{2^t})^{s-3}(\frac{B_t}{2^{t-1}})^3\cdot\frac{r}{4s}.
\end{eqnarray}
Noting that (\ref{integer}) implies
$$\frac{A_t}{2^t}=(\frac{p}{2})^t+c_1(p,r,t)\quad\text{where }\frac{r}{4}\mid c_1(p,r,t)$$
and (\ref{Gn expan}) implies
$$\frac{B_t}{2^{t-1}}=t(\frac{p}{2})^{t-1}+c_2(p,r,t)\quad\text{where }\frac{r}{4}\mid c_2(p,r,t),$$
by $s\mid\frac{r}{4}$ and (\ref{s div r}) we get
$$s\mid(\frac{p}{2})^{t(s-1)}\cdot t(\frac{p}{2})^{t-1}+\binom{s}{3}(\frac{p}{2})^{t(s-3)}(t(\frac{p}{2})^{t-1})^3\cdot\frac{r}{4s},$$
$$\text{i.e.,}\quad s\mid(\frac{p}{2})^{st-3}\Big(t(\frac{p}{2})^2+t^3\cdot\frac{r}{4}\cdot\frac{(s-1)(s-2)}{6}\Big).$$
It follows from $(s,\frac{p}{2})=1$ that
\begin{eqnarray}\label{s div t}
s\mid t(\frac{p}{2})^2+t^3\cdot\frac{r}{4}\cdot\frac{(s-1)(s-2)}{6}.
\end{eqnarray}
In a similar way as the end of \textcircled{\footnotesize{1}}, the contradiction follows.
\newline\textcircled{\footnotesize{3}} Suppose that $(p,q)=1$, $s\ge3$ is a prime satisfying $s\mid r$, and $3\nmid q+1$ or $3\nmid s$. If $p$ is odd, the conclusion follows immediately from \textcircled{\footnotesize{1}}. If $p$ is even, in the same way as the proof of Theorem \ref{main2} (\underline{Case 3}), we get $(\frac{p}{2},q)=1$ and $s\mid\frac{r}{4}$ by the condition $s\ge3$. Then the conclusion follows immediately from \textcircled{\footnotesize{2}}.
\end{proof}

\begin{proof}[Proof of Corollary \ref{square}] Suppose that
$$p\text{ is odd, }(p,q)=1\text{ and }s^2\mid r$$
or
$$p\text{ is even, }(\frac{p}{2},q)=1\text{ and }s^2\mid\frac{r}{4}.$$
By Theorem \ref{main2} (1), we know that for all integers $n\ge0$,
\begin{eqnarray}\label{square iff}
s^2\mid n\quad\text{if and only if}\quad s^2\mid G_n.
\end{eqnarray}
In order to complete the proof, it suffices to check the condition of Theorem \ref{main2} (2). In fact, for all $t\in\N$ satisfying $s\nmid t$, we have $s^2\nmid st$. It follows from (\ref{square iff}) that $s^2\nmid G_{st}$.
\end{proof}

Corollary \ref{classical} (1), (2) and (3) \textcircled{\footnotesize{1}} follow from Theorem \ref{main}, while (3) \textcircled{\footnotesize{2}} follows from Corollary \ref{square}.

Corollary \ref{q=1} follows immediately from taking $q=1$ in Theorem \ref{main} (2).

\begin{proof}[Proof of Corollary \ref{p=1 2}] Let $q\in\Z$ and $s\in\N$. If $(G_n)_{n\ge0}$ is the $<1,q>$-Fibonacci sequence and $s\mid4q+1$, or $(G_n)_{n\ge0}$ is the $<2,q>$-Fibonacci sequence and $s\mid q+1$ with $q\neq-1$, then the conclusions follow directly from Theorem \ref{main2} (1) and Theorem \ref{main} (2). If $(G_n)_{n\ge0}$ is the $<2,-1>$-Fibonacci sequence, it is straightforward to get $G_n=n$ for all $n\ge0$, and then the conclusions follow.
\end{proof}

\begin{proof}[Proof of Corollary \ref{r}] Let $p\in\Z$, $q\in\N$ and $r=p^2+4q\ge4$.
\newline(1) Suppose that $r$ is a prime. Then we have the following.
\begin{itemize}
\item[\textcircled{\footnotesize{1}}] $p$ is odd, since if $p$ is even, then $4\mid r$ will contradict that $r$ is a prime;
\item[\textcircled{\footnotesize{2}}] $(p,q)=1$, since $(p,q)\mid r$, $r$ is a prime and $(p,q)\le q<4q\le r$.
\item[\textcircled{\footnotesize{3}}] $3\nmid r$, since $r$ is a prime and $r\ge4$.
\end{itemize}
By taking $s=r$ in Theorem \ref{main} (2), the conclusion follows.
\newline(2) Suppose that $4\mid r$, $\frac{r}{4}=(\frac{p}{2})^2+q$ is a prime and $p\neq0$.
\begin{itemize}
\item[\textcircled{\footnotesize{1}}] $p$ is even, since $4\mid r$.
\item[\textcircled{\footnotesize{2}}] $(\frac{p}{2},q)=1$, since $(\frac{p}{2},q)\mid\frac{r}{4}$, $\frac{r}{4}$ is a prime and $(\frac{p}{2},q)\le q<\frac{r}{4}$.
\end{itemize}
Let $s=\frac{r}{4}$. If $3\nmid s$, the conclusion follows immediately from \textcircled{\footnotesize{1}}, \textcircled{\footnotesize{2}} and Theorem \ref{main} (2). We only need to consider $3\mid s$ in the following. Since $s$ is a prime, we get $s=3$. It follows from $s=(\frac{p}{2})^2+q$, $p\in\Z\setminus\{0\}$ and $q\in\N$ that $(\frac{p}{2})^2=1$ and $q=2$. In order to complete the proof, by Theorem \ref{main2} (2), it suffices to check that for all $t\in\N$ satisfying $3\nmid t$, we have $3^2\nmid G_{3t}$.
\newline(By contradiction) Assume $3^2\mid G_{3t}$. In the same way as the proof of Theorem \ref{main} (2) \textcircled{\footnotesize{2}}, we get (\ref{s div t}). That is, $3\mid t+t^3$. By $3\nmid t$, there exists integer $m\ge0$ such that $t=3m+1$ or $3m+2$. Thus $3\mid(3m+1)+(3m+1)^3$ or $(3m+2)+(3m+2)^3$, which implies $3\mid1+1$ or $2+8$. This is impossible.
\end{proof}

\section{Further questions}

In Theorem \ref{main}, \ref{main2} and Corollary \ref{square}, \ref{q=1}, \ref{p=1 2}, we give sufficient conditions (on $p,q,r,s$) for the equivalences of
$$s\mid n\quad\text{and}\quad s\mid G_n\quad\text{for all integers }n\ge0,$$
and
$$s^k\mid n\quad\text{and}\quad s^k\mid G_n\quad\text{for all integers }k,n\ge0.$$
What are the necessary and sufficient conditions for these equivalences?

\medskip

\noindent MSC2010: 11B39

\end{document}